\documentclass[10pt,reqno,oneside]{amsart}
\usepackage{verbatim,amsmath,amssymb,cite,epsfig,subfigure,color}

\numberwithin{equation}{section}

\newcommand{\Real}{\mathbb R}
\newcommand{\N}{\mathbb N}

\newcommand{\C}{\mathbb C}

\renewcommand\Re{\operatorname{Re}}
\renewcommand\Im{\operatorname{Im}}

\newtheorem{theorem}{Theorem}[section]

\newtheorem{proposition}{Proposition}[section]
\newtheorem{remark}{Remark}[section]
\newtheorem{corollary}{Corollary}[section]

\begin{document}

\title[Laguerre's method applied to find roots of unity]
{Analysis of Laguerre's method applied to find the roots of unity}

\author{Pavel B\v{e}l\'{\i}k}
\author{HeeChan Kang}
\author{Andrew Walsh}
\author{Emma Winegar}

\address{P.~B\v{e}l\'{\i}k\\
Mathematics Department\\
Augsburg College\\
2211 Riverside Avenue\\
Minneapolis, MN 55454\\
U.S.A.}
\email{belik@augsburg.edu}

\address{H.~Kang\\
Augsburg College\\
2211 Riverside Avenue\\
Minneapolis, MN 55454\\
U.S.A.}
\email{kang@augsburg.edu}

\address{A.~Walsh\\
Augsburg College\\
2211 Riverside Avenue\\
Minneapolis, MN 55454\\
U.S.A.}
\email{walsha@augsburg.edu}

\address{E.~Winegar\\
Augsburg College\\
2211 Riverside Avenue\\
Minneapolis, MN 55454\\
U.S.A.}
\email{winegar@augsburg.edu}

\thanks{
  This work was supported in part by the National Science Foundation Grant DMS-0802959 and in part by The Office of Undergraduate Research and Graduate Opportunity at Augsburg College.
}



\date{\today}

\begin{abstract}
Previous analyses of Laguerre's method have provided results on the convergence and properties of this popular method when applied to the polynomials $p_n(z)=z^n-1$, $n\in\N$ \cite{ray66,curryfiedler88,drakopoulos03}. While these analyses appear to provide a fairly complete picture, careful study of the results reveals that more can be said. We provide additional analytical, computational, and graphical results, details, and insights. We raise and summarize questions that still need to be answered.
\end{abstract}

\maketitle{\allowdisplaybreaks\thispagestyle{empty}

\section{Introduction}
Laguerre's method for approximating roots of polynomials \cite{laguerre} is one of the least understood methods of numerical analysis. It exhibits cubic convergence to simple roots of (complex) polynomials and linear convergence to multiple roots, thus outperforming the well-known Newton's method that exhibits quadratic convergence to simple roots \cite{parlett64}, or even the widely used, and globally convergent, Jenkins--Traub method, which has the order of convergence of (at least) $1+\phi$, where $\phi=(1+\sqrt{5})/2$ is the golden ratio \cite{ralstonrabinowitz01,ford77}.
Perhaps due to the lack of complete understanding of Laguerre's method, it is often overlooked in designing professional software. However, some of the known results make it an excellent candidate in many situations. For example, it is known that the method exhibits global convergence (convergence from any initial guess) for real polynomials with real roots \cite{bodewig46,ralstonrabinowitz01}. It also allows for automatic switching to the complex domain if there are no real roots; this is due to the appearance of a square root in the definition of the method (see \eqref{eq:laguerre} in the next section). In general, although convergence is not guaranteed for all complex starting values, the method seems to perform very well in many cases \cite{parlett64}.

It is the goal of this paper to provide additional insights into the performance of Laguerre's method when applied to simple polynomials of the form $z^n-1$. We primarily follow the work of Ray \cite{ray66} and Curry and Fiedler \cite{curryfiedler88}, but provide additional details and clear proofs of all results. In addition, we provide computational results that demonstrate the poor performance of the method when $n$ is large and exact arithmetic is used. This is due to the fact that the region of convergence to the roots is contained in an annulus that shrinks towards the unit circle $S^1$ as $n$ increases. Points in the complement of the annulus converge to a two-cycle consisting of $\{0,\infty\}$. We also show that the boundary of the region of convergence has fractal characteristics and becomes quite interesting for large $n$. Finally, we demonstrate on some examples that in floating-point arithmetic the method in its general formulation \eqref{eq:laguerre} eventually converges from seemingly any initial complex value due to the loss of significance. This thus ironically contributes to the practicality of the method in this case, and it remains to be seen whether this is also the case for general polynomials.

We organize the paper similarly as in \cite{curryfiedler88}. In section \ref{sec:laguerre}, we introduce the method, briefly summarize known results, and provide a simpler expression for the method when applied to the polynomials $p_n(z)=z^n-1$. In section \ref{sec:symmetries}, we formulate and prove three propositions regarding the symmetries of the method applied to $p_n$ that simplify the analysis in the following sections. Regions in the complex plane that will play a significant role in the study of the dynamics are defined in section \ref{sec:regions}, and their boundary curves are algebraically characterized in section \ref{sec:char_fun}. The dynamics of the method on the unit circle and in the neighborhood of the two-cycle $\{0,\infty\}$ is studied in sections \ref{sec:unit_circle} and \ref{sec:two_cycle}, respectively. In section \ref{sec:convergence} we provide proofs of convergence to the roots of unity when the initial guess is in a relevant annulus containing the unit circle. The boundary of the region of convergence is contained in two annuli shown as the ``gray areas'' in Fig.~\ref{fig:regions}, and some relevant numerical results pertinent to the boundary are shown in section \ref{sec:boundary}. We conclude with section \ref{sec:conclusions}, in which we summarize some of the open questions, and demonstrate the ``convergence'' of the method even from the basins of attraction of the two-cycle $\{0,\infty\}$.

\section{Laguerre's method}
\label{sec:laguerre}
In this section, we provide the basic details of Laguerre's method, mention known results, and apply the method to the polynomials $p_n(z)=z^n-1$ with $n\ge2$ and $z\in\C$. We will denote by $z^{1/2}$ the set of the two solutions $\{w,-w\}\subset\C$ such that $w^2=z$ (unless, of course, $z=0$, in which case $w=0$). We will use the notation $\sqrt{z}$ for the principal square root of $z$; i.e., if $z=re^{i\theta}$ with $r>0$ and $-\pi<\theta\le\pi$, then $\sqrt{z}=\sqrt{r}\,e^{i\theta/2}$.

Laguerre's method for complex polynomials $p(z)$ of degree $n\ge2$ is defined as \cite{laguerre,ralstonrabinowitz01}
\begin{equation}
  \label{eq:method}
  z_{k+1}
  =
  L(z_k)
  \quad
  (z_0\in\C\text{ given}),
\end{equation}
where $L(z)$ denotes the {\it Laguerre iteration function} \cite{henrici74} given by
\begin{equation}
  \label{eq:laguerre}
  \begin{split}
    L(z)
    &=
    z
    -
    \frac{np(z)}{p'(z)\pm\sqrt{(n-1)^2\left(p'(z)\right)^2-n(n-1)p(z)p''(z)}}\\
    &=
    z-\frac{n}{G(z)\pm\sqrt{(n-1)(nH(z)-G^2(z))}},
  \end{split}
\end{equation}
where
\begin{equation*}
  G(z)
  =
  \frac{p'(z)}{p(z)}
  \quad\text{and}\quad
  H(z)
  =
  G^2(z)-\frac{p''(z)}{p(z)},
\end{equation*}
and where the sign is chosen so as to maximize the modulus of the denominators.

\subsection{Known Results}
It is known that Laguerre's method exhibits cubic convergence to a simple root and linear convergence to a multiple root \cite{parlett64,ralstonrabinowitz01}. It is also known that for a real polynomial with real roots, the method converges to a root from any initial guess $z_0\in\Real$ \cite{ralstonrabinowitz01}. A particular feature of interest is that even if the initial guess is a real number, convergence to a complex root can occur due to the square root in the denominator of \eqref{eq:laguerre}. In many cases, the method seems to converge to a root from any initial guess in the complex plane, although this is not the case in general \cite{ray66}. For example, consider the polynomial $p_n(z)=z^n-1$ with $n\ge3$, for which both the first and the second derivative vanish at $z=0$, and $L(0)$ is undefined (in what follows, we will consider the extended complex plane $\hat\C=\C\cup\{\infty\}$ so that $L(0)=\infty$ and $L(\infty)=0$, and $\{0,\infty\}$ forms a two-cycle of the method). It is also known \cite{kahan67} that if $p(z)$ is a polynomial of degree $n$ and $z\in\C$, then there exists a root $z^*$ of $p$ such that $|z-z^*|\le\sqrt{n}\,|z-L(z)|$.

The Laguerre iteration function \eqref{eq:laguerre} is  sometimes claimed to be invariant under M\"{o}bius transformations \cite{drakopoulos03,parlett64}, although the correct, weaker statement is given and proved in \cite{ray66}. For the classes of quadratics and cubics of the form $p_c(z)=z^2+c$ and $p_\lambda(z)=(z-1)(z^2+z+\lambda)$, respectively, with $c,\lambda\in\C$, the Laguerre iteration function \eqref{eq:laguerre} can be shown to not have any free critical points \cite{drakopoulos03}, and a generalization to all complex quadratics and cubics is suggested based on the invariance under M\"{o}bius transformations.

\subsection{Roots of Unity}
When applied to the polynomial $p_n(z)=z^n-1$, $n\ge2$, the Laguerre iteration function \eqref{eq:laguerre} for $z\ne0$ simplifies to
\begin{equation*}
  L(z)
  =
  z\,\frac{z^{-n/2}\pm(n-1)}{z^{n/2}\pm(n-1)},
\end{equation*}
where again the sign is chosen to maximize the modulus of the denominator. Using the principal square root of $z^n$, which will result in an expression with a nonnegative real part, we can rewrite $L(z)$ as $L_p(z)$ given by
\begin{equation}
  \label{eq:Lp}
  L_p(z)
  =
  z\,\frac{\frac{1}{\sqrt{z^n}}+(n-1)}{\sqrt{z^n}+(n-1)}.
\end{equation}

We note that the roots of $p_n$ are exactly the fixed points of $L_p$, and it is easy to check that the derivative of $L_p$ vanishes at the roots, so they are attracting fixed points, and each has an open neighborhood contained in its basin of attraction.

It is straightforward to check that in the case $n=2$ the method converges in one iteration for any initial guess $z_0\in\C$. If $\Re(z_0)>0$, or if $\Re(z_0)=0$ and $\Im(z_0)\ge0$, then $L_p(z_0)=1$; otherwise $L_p(z_0)=-1$.

In the cases with $n\ge3$, the method has a two-cycle consisting of $0$ and $\infty$ in the extended complex plane $\hat\C$. Other than this two-cycle, the method is globally convergent for $n=3,4$ \cite{ray66,curryfiedler88}.

The behavior of Laguerre's method is quite different in the cases with $n\ge5$ and is the subject of our interest. In the following sections, we closely follow the analysis of Curry and Fiedler \cite{curryfiedler88}, which in turn is based on the work of Ray \cite{ray66}. We provide additional insights and graphical illustrations for some of the results. While the main focus will be on the cases with $n\ge5$, if a result applies more generally, we will state so.

For future reference, we note that for $r>0$, we have \cite{curryfiedler88}
\begin{equation}
  \label{eq:mag_Lp2}
  |L_p(re^{i\theta})|^2
  =
  \frac{1}{r^{n-2}}
  \left(
    \frac{1+(n-1)^2r^n+2(n-1)\left|\cos\frac{n\theta}{2}\right|r^{n/2}}{r^n+(n-1)^2+2(n-1)\left|\cos\frac{n\theta}{2}\right|r^{n/2}}
  \right).
\end{equation}

\section{Symmetries of Laguerre's method}
\label{sec:symmetries}
When applied to the polynomials $p_n(z)=z^n-1$, $n\ge2$, the Laguerre iteration function \eqref{eq:Lp} exhibits several symmetries that simplify the analysis of the method in the extended complex plane $\hat\C$.
In particular, the method possesses an $n$-fold rotational symmetry around the origin, a symmetry with respect to the real axis, and also an inversion symmetry with respect to the unit circle. In Propositions~\ref{prop:rotational_symmetry}--\ref{prop:inversion_symmetry} we provide the precise statements.

We will use the slightly imprecise notation of \cite{curryfiedler88} and denote by $\theta_0$ and $\theta_1$ any of the following angles for $k=0,\dots,n-1$:
\begin{equation}
  \label{eq:theta01}
  \theta_0
  =
  \frac{2k\pi}{n}
  \qquad\text{and}\qquad
  \theta_1
  =
  \frac{(2k+1)\pi}{n}.
\end{equation}
In addition, we define the rays
\begin{equation}
  \label{eq:theta_rays}
  \Theta_0
  =
  \{re^{i\theta_0}\in\C:\ r>0\}
  \qquad\text{and}\qquad
  \Theta_1
  =
  \{re^{i\theta_1}\in\C:\ r>0\}.
\end{equation}
Note that the roots of $p_n$ lie on the rays $\Theta_0$, while the rays $\Theta_1$ divide the complex plane into $n$ congruent sectors bisected by the rays $\Theta_0$.

The following proposition implies that it suffices to study the behavior of $L_p$ in the sector $\{z\in\C:\ -\pi/n<\arg{z}\le\pi/n\}$, i.e., between two consecutive rays $\Theta_1$, and the rest follows by rotational symmetry. This is a special case of the invariance of the method with respect to certain M\"{o}bius transformations \cite{parlett64,ray66}.
\begin{proposition}
  \label{prop:rotational_symmetry}
  For $n\ge2$, the Laguerre iteration function $L_p$ defined in \eqref{eq:Lp} commutes with the rotation by an angle $\alpha=2\pi/n$ about the origin.
\end{proposition}
\begin{proof}
  Let $T_\alpha$ denote the rotation by an angle $\alpha=2\pi/n$ about the origin, i.e., $T_\alpha(z)=e^{i\alpha}z$. Since $\left(T_\alpha(z)\right)^n=z^n$, substituting into \eqref{eq:Lp}, we get for any $z\in\hat\C$
  \begin{align*}
    L_p(T_\alpha(z))
    =
    e^{i\alpha}z\,\frac{\frac{1}{\sqrt{z^n}}+(n-1)}{\sqrt{z^n}+(n-1)}
    =
    e^{i\alpha}L_p(z)
    =
    T_\alpha\left(L_p(z)\right),
  \end{align*}
  and the result follows.
\end{proof}

The following proposition implies that the behavior of $L_p$ in the sector $\{z\in\C:\ -\pi/n<\arg{z}<\pi/n\}$ is symmetric with respect to the real axis. In the case when $\arg(z)=\theta_1$, Proposition \ref{prop:rotational_symmetry} applies.
\begin{proposition}
  \label{prop:conjugation_symmetry}
  For $n\ge2$ and $z\in\C$ with $\arg{z}\ne\theta_1$, the Laguerre iteration function $L_p$ defined in \eqref{eq:Lp} commutes with the complex conjugation $z\mapsto\bar{z}$.
\end{proposition}
\begin{proof}
  If $\arg{z}\ne\theta_1$, then 
  $\sqrt{(\bar{z})^n}=\overline{\sqrt{z^n}}$. Consequently, $L_p(\bar{z})=\overline{L_p(z)}$, and the result follows.
\end{proof}

Finally, Laguerre's iteration function \eqref{eq:Lp} also exhibits inversion symmetry with respect to the unit circle.
\begin{proposition}
  \label{prop:inversion_symmetry}
  For $n\ge2$, the Laguerre iteration function $L_p$ defined in \eqref{eq:Lp} commutes with the inversion with respect to the unit circle $S^1=\{z\in\C:\ |z|=1\}$.
\end{proposition}
\begin{proof}
  Consider first $z\in\C$ with $\arg{z}\ne\theta_1$. Using the same conjugation properties as in the proof of Proposition \ref{prop:conjugation_symmetry}, we have
  \begin{equation*}
    L_p\left(1/\bar{z}\right)
    =
    \frac{1}{\bar{z}}\,\frac{1/\sqrt{\left(1/\bar{z}\right)^n}+(n-1)}{\sqrt{\left(1/\bar{z}\right)^n}+(n-1)}
    =
    \frac{1}{L_p(\bar{z})}
    =
    \frac{1}{\overline{L_p(z)}}.
  \end{equation*}
  We can then check by direct substitution that the same result holds also when $\arg(z)=\theta_1$, and the conclusion of the proposition follows.
\end{proof}

\section{Regions of significance in the complex plane}
\label{sec:regions}
Consider from now on the polynomial $p_n(z)=z^n-1$ with $n\ge5$ and the corresponding Laguerre iteration function $L_p$ given by \eqref{eq:Lp}. Following \cite{curryfiedler88}, we start by defining several regions in the extended complex plane $\hat\C$ relevant for the study of the dynamics of Laguerre's method. We will provide relevant results, some of which are proved in \cite{curryfiedler88}.

It is stated in \cite{curryfiedler88} that the regions in \eqref{eq:regions} below ``contain all the dynamics'' of \eqref{eq:Lp}. This is not quite true, although these regions are of significance in the analysis. They divide $\C\setminus\{0\}$ into disjoint subsets and are defined as
\begin{equation}
  \label{eq:regions}
  \begin{split}
    D
    &=
    \{z\in\C:\ 0<|z|<1\text{ and }|L_p(z)|>1/|z|\},\\
    \partial D
    &=
    \{z\in\C:\ 0<|z|<1\text{ and }|L_p(z)|=1/|z|\},\\
    K_0
    &=
    \{z\in\C:\ 0<|z|<1\text{ and }|L_p(z)|<1/|z|\},\\
    S^1
    &=
    \{z\in\C:\ |z|=1\},\\
    K_1
    &=
    \{z\in\C:\ |z|>1\text{ and }|L_p(z)|>1/|z|\},\\
    \partial E
    &=
    \{z\in\C:\ |z|>1\text{ and }|L_p(z)|=1/|z|\},\\
    E
    &=
    \{z\in\C:\ |z|>1\text{ and }|L_p(z)|<1/|z|\}.
  \end{split}
\end{equation}
Note that, due to the use of the principal square root in \eqref{eq:Lp}, $L_p$ is continuous everywhere in $\C\setminus\{0\}$ except across the rays $\Theta_1$; however, $|L_p|$ is continuous across $\Theta_1$, so the notation $\partial D$ and $\partial E$ is justified, since $\partial D$ and $\partial E$ are the boundaries of $D$ and $E$, respectively, in $\C\setminus\{0\}$. See Fig.~\ref{fig:regions} for an illustration of the sets in \eqref{eq:theta_rays} and \eqref{eq:regions} for $n=16$; the cases with other values of $n$ are similar. For future reference we note that $L_p$ is ``counter-clockwise'' continuous across the rays $\Theta_1$. That is, if $z\to re^{i\theta_1}$ with $\arg{z}<\theta_1$, then $L_p(z)\to L_p(re^{i\theta_1})$. This is not the case in the ``clockwise'' direction.

\begin{figure}
  \includegraphics[width=0.6\textwidth]{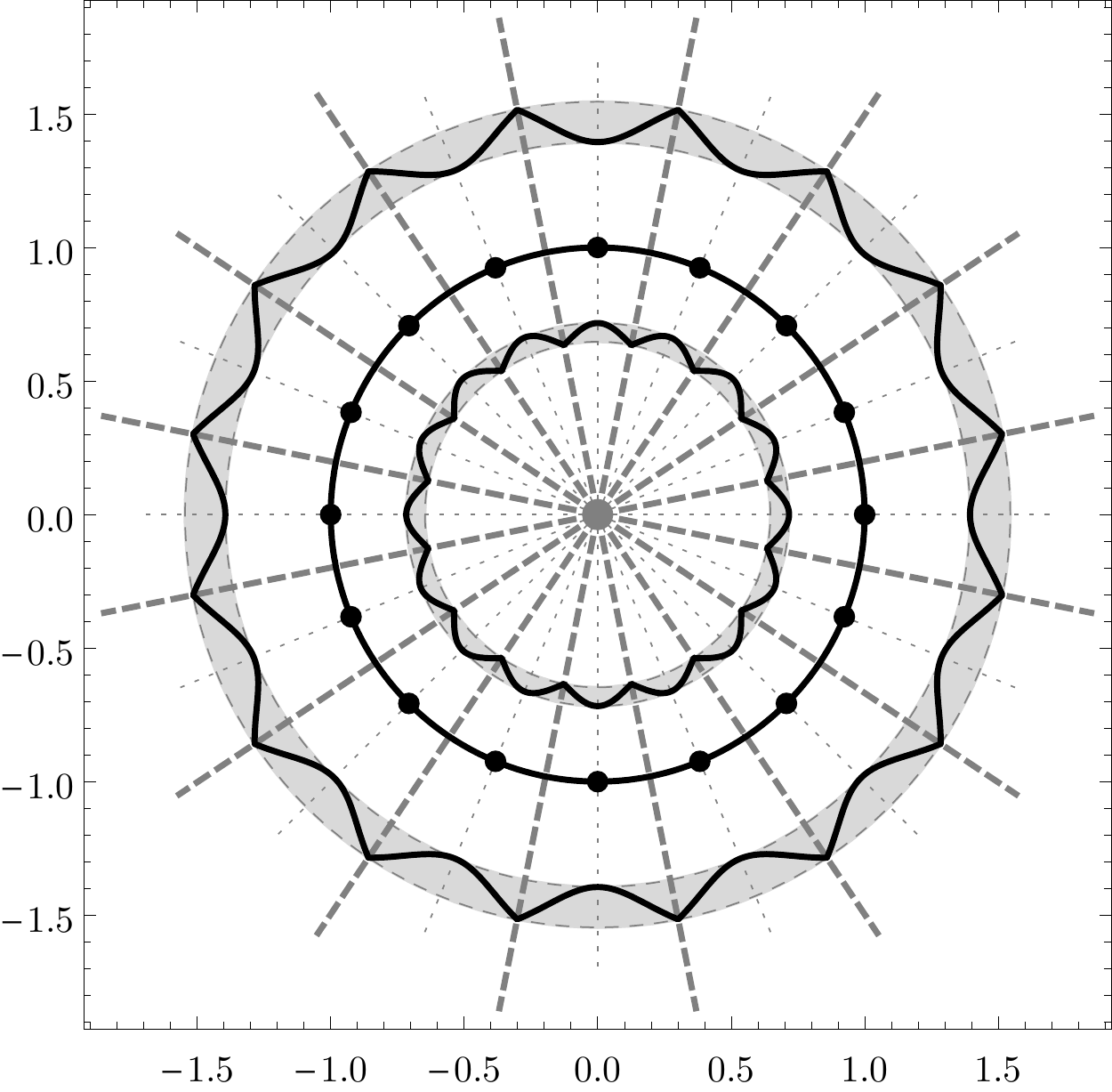}
  \caption{Illustration of the regions of significance for $p_n(z)=z^n-1$ with $n=16$ defined in \eqref{eq:theta_rays} and \eqref{eq:regions}. The thick solid curves correspond to the solutions of equation \eqref{eq:f=0} and are, in the order of increasing distance from the origin, $\partial D$, $S^1$, and $\partial E$. The dots indicate the position of the roots of $p_n$. The thick dashed lines are the rays $\Theta_1$, while the thin dotted lines are the rays $\Theta_0$. The thin dashed circles have radii $s_0<r_0<1/r_0<1/s_0$ as defined in \eqref{eq:r0s0}. Finally, the open region between $\partial D$ and $S^1$ is $K_0$, and the open region between $S^1$ and $\partial E$ is $K_1$.}
  \label{fig:regions}
\end{figure}

\section{The characteristic function for $\partial D$, $S^1$, and $\partial E$}
\label{sec:char_fun}
As in \cite{ray66,curryfiedler88}, we now focus on the algebraic characterization of $\partial D$ and $\partial E$. This will lead to a definition and study of a ``characteristic function'' \eqref{eq:f} below that will allow us to determine the shapes of the boundary curves as shown in Fig.~\ref{fig:regions}.

Writing $z=re^{i\theta}$, $0<r<\infty$, we note that both $\partial D$ and $\partial E$ are characterized by the same equation,
\begin{equation}
  \label{eq:bDbE}
  |L_p(z)|
  =
  1/|z|
  \qquad\text{or}\qquad
  |L_p(re^{i\theta})|
  =
  \frac{1}{r}.
\end{equation}
Using \eqref{eq:mag_Lp2}, equation \eqref{eq:bDbE} is equivalent to \cite{curryfiedler88,ray66}
\begin{equation}
  \label{eq:f=0}
  f_n(r,\theta)
  =
  0,
\end{equation}
where the ``characteristic function'' $f_n$ is defined as
\begin{equation}
  \label{eq:f}
  f_n(r,\theta)
  =
  r^{2n-4}
  +
  2(n-1)\left|\cos\frac{n\theta}{2}\right|r^{n/2}(r^{n-4}-1)
  -
  (n-1)^2r^n
  +
  (n-1)^2r^{n-4}
  -
  1.
\end{equation}

We summarize relevant results (some stated in \cite{curryfiedler88}) in the following theorem.
\begin{theorem}
  \label{thm:f_properties}
  Let $n\ge5$, $z=re^{i\theta}$ with $0<r<\infty$, and let $f_n(r,\theta)$ be defined as in \eqref{eq:f}. We then have the following.
  \begin{enumerate}
    \item For every $\theta\in\Real$, equation \eqref{eq:f=0} has exactly three positive zeroes, $r_D$, $1$, and $r_E$, such that $r_D<1<r_E=1/r_D$. In addition, $r_E<(n-1)^{2/(n-4)}$, so the zeroes converge to $1$ as $n\to\infty$.
    \item Each of the regions in \eqref{eq:regions} corresponds to a particular sign of $f_n$:
      \begin{align*}
        f_n(r,\theta)
	=
	0
        &\quad\Leftrightarrow\quad
        |L_p(z)|
        =
        1/|z|
        \quad\Leftrightarrow\quad
        z\in\partial D\cup S^1\cup\partial E,\\
        f_n(r,\theta)
	<
	0
	&\quad\Leftrightarrow\quad
        |L_p(z)|
        >
        1/|z|
        \quad\Leftrightarrow\quad
        z\in D\cup K_1,\\
        f_n(r,\theta)
	>
	0
	&\quad\Leftrightarrow\quad
	|L_p(z)|
	<
	1/|z|
	\quad\Leftrightarrow\quad
	z\in K_0\cup E.
	\end{align*}
    \item The boundaries $\partial D$ and $\partial E$ correspond to polar curves of the form $r=r_D(\theta)$ and $r=r_E(\theta)$. The function $r_D(\theta)$ is maximized at any $\theta=\theta_0$ and minimized at any $\theta=\theta_1$, while the function $r_E(\theta)$ is minimized at any $\theta=\theta_0$ and maximized at any $\theta=\theta_1$. In addition, both $r_D(\theta)$ and $r_E(\theta)$ are monotonic between any two consecutive angles $\theta_0$ and $\theta_1$.
    (See Fig.~\ref{fig:regions}.)
  \end{enumerate}
\end{theorem}
\begin{proof}
  (1) Let $n\ge5$, $\theta\in\Real$, and define $f(r)=f_n(r,\theta)$. Note that $f$ is a differentiable function and
  \begin{equation}
    \label{eq:f_properties}
    f(0)=-1,
    \qquad
    f(1)=0,
    \qquad
    \lim_{r\to+\infty}f(r)=+\infty,
    \qquad\text{and}\qquad
    f'(1)<0.
  \end{equation}
  (It is easy to show that $f'(1)\le-2n^2$.) This implies that $f$ has at least three positive zeroes. From \eqref{eq:bDbE} and Proposition \ref{prop:inversion_symmetry} it follows that, other than $1$, the zeroes of $f$ come in reciprocal pairs, so the actual number of zeroes is an odd number greater than or equal to $3$. As in \cite{curryfiedler88,ray66}, we will invoke Descartes' rule of signs. When $n$ is even, $f$ is a polynomial, so the rule can be applied directly. When $n$ is odd, we can apply it to $g(R)=f(R^2)$, which is a polynomial that also satisfies \eqref{eq:f_properties} with $f(r)$ replaced by $g(R)$. Hence, we focus on $f$ with the understanding that $g$ is handled exactly the same way. Note that $f$ can be expanded to contain at most $6$ terms with different powers of $r$, hence there are at most $5$ sign changes, and $f$ has at most $5$ positive zeroes. From \eqref{eq:f_properties} it now follows that if $f$ had $5$ zeroes, two of them would have to have multiplicity greater than $1$ and $f'$ would have to have at least $6$ positive zeroes ($4$ between the zeroes of $f$ and at least $2$ more from the multiple roots of $f$). This is, however, impossible, since $f'$ is another polynomial with at most $5$ terms of different powers of $r$, hence Descartes's rule of signs implies $f'$ has at most $4$ positive zeroes. Consequently, $f$ has exactly $3$ simple positive zeroes as stated in the theorem.

  Finally, a straightforward computation with $(n-1)^{2/(n-4)}>1$ shows that
  \begin{equation*}
    f\left((n-1)^{2/(n-4)}\right)
    =
    (n-1)^4-1+2n(n-2)(n-1)^{(2n-4)/(n-4)}\left|\cos(n\theta/2)\right|
    >
    0,
  \end{equation*}
  so, since $f$ is negative for $1<r<r_E$ and positive for $r>r_E$, we have that $(n-1)^{2/(n-4)}>r_E$. Application of L'H\^{o}pital's rule shows that $(n-1)^{2/(n-4)}\to1$ as $n\to\infty$.

  (2) This part follows from the definition of the relevant regions in \eqref{eq:regions} and from replacing the equality in \eqref{eq:bDbE} by inequalities, which results in inequalities in \eqref{eq:f=0} \cite{curryfiedler88}.

  (3) Since for every $\theta\in\Real$ there are unique values of $r_D$ and $r_E$, we can think of $\partial D$ and $\partial E$ as polar curves. To prove all of the remaining statements in this part, it is enough to consider $r_E(\theta)$ for $0\le\theta\le\pi/n$, since the rest follows by the symmetries discussed earlier. Note that the cosine term in \eqref{eq:f} is largest for $\theta=0$, so on the circle $r=r_E(0)>1$, as a function of $\theta$, $f_n(r_E(0),\theta)$ is largest (and equal to $0$) exactly when $\theta=\theta_0$. Thus, $f_n$ is negative on the circle for every $\theta\ne\theta_0$, and it follows that $r_E(\theta)\ge r_E(\theta_0)$ for all $\theta\in\Real$. Similarly, the cosine term is smallest when $\theta=\pi/n$, and by a similar argument we get $r_E(\theta)\le r_E(\theta_1)$ for all $\theta\in\Real$. Finally, to prove the last assertion, implicitly differentiate \eqref{eq:f=0} with respect to $\theta$ and observe that $dr_E/d\theta=-(\partial f_n/\partial\theta)/(\partial f_n/\partial r)$ vanishes only when $\theta=\theta_0$ and does not exist only when $\theta=\theta_1$, since the numerator contains a factor $\sin(n\theta/2)$ and the denominator is positive on $r_E(\theta)$ as the proof of part (1) implies. This concludes the proof of the theorem.
\end{proof}

\begin{remark}
  As in \cite{curryfiedler88}, we define the values $0<s_0<r_0<1$ by
  \begin{equation}
    \label{eq:r0s0}
    s_0
    =
    r_D(\theta_1)
    =
    \min_{0\le\theta<2\pi}{r_D(\theta)}
    \qquad\text{and}\qquad
    r_0
    =
    r_D(\theta_0)
    =
    \max_{0\le\theta<2\pi}{r_D(\theta)}.
  \end{equation}
  The four circles with radii $s_0<r_0<1/r_0<1/s_0$ are shown in Fig.~\ref{fig:regions} as dashed circles, and the annuli $\{s_0<r<r_0\}$ and $\{1/r_0<r<1/s_0\}$ are shaded gray.
\end{remark}


\section{Dynamics on the unit circle}
\label{sec:unit_circle}
In this section, we study the dynamics on the unit circle, $S^1$. As a consequence of Theorem \ref{thm:f_properties}, part (1), we have that the unit circle, $S^1$, is invariant under the Laguerre iteration function (see also \cite{curryfiedler88}). This follows from \eqref{eq:f=0} and \eqref{eq:bDbE} with $r=1$.
However, more can be said about the behavior of $L_p$ on $S^1$.
\begin{proposition}
  \label{prop:unit_circle}
  If $z_0\in S^1\setminus\Theta_1$, then the sequence if iterates of Laguerre's method, $\{L_p^k(z_0)\}$, converges monotonically to the nearest root of $p_n(z)=z^n-1$ in the sense that $|L_p^k(z_0)|=1$ and the arguments of $L_p^k(z_0)$ monotonically approach the argument of the nearest root. If $z_0\in S^1\cap\Theta_1$, then the iterates converge monotonically to the nearest root of $p_n$ in the clockwise direction.
\end{proposition}
\begin{proof}
  Due to the symmetries of $L_p$ (Propositions \ref{prop:rotational_symmetry} and \ref{prop:conjugation_symmetry}), it is enough to assume $z=e^{i\theta}$ with $0<\theta\le\pi/n$ and show that $L_p(z)=e^{i\tilde\theta}$ with $0<\tilde\theta<\theta$. Since the sequence of arguments generated by the method will then be a decreasing sequence bounded below by $0$, it will converge, and the corresponding sequence of points on $S^1$ will converge to a fixed point of $L_p$, i.e., to a root of $p_n$ by the continuity of $L_p$ away from $\Theta_1$. The limiting root will have to be $1$ and both assertions of the proposition follow.

  To prove that $0<\tilde\theta<\theta$, we first note that
  \begin{equation*}
    L_p(z)
    =
    e^{i\theta}\,\frac{e^{-i\frac{n\theta}{2}}+(n-1)}{e^{i\frac{n\theta}{2}}+(n-1)},
  \end{equation*}
  where the numerator and the denominator are conjugates of each other and $n\theta/2\le\pi/2$. Thus the whole fraction results in an expression of the form $e^{-2i\hat\theta}$ with $0<\hat\theta<\pi/2$, where
  \begin{equation*}
    \hat\theta=\arg\left(e^{i\frac{n\theta}{2}}+(n-1)\right)=\arctan\left(\frac{\sin\frac{n\theta}{2}}{\cos\frac{n\theta}{2}+(n-1)}\right).
  \end{equation*}
  Consequently, $L_p(z)=e^{i(\theta-2\hat\theta)}$ and it remains to show that $\theta-2\hat\theta>0$. Consider the function
  \begin{equation*}
    g(\theta)
    =
    \theta
    -
    2\hat\theta
    =
    \theta
    -
    2\arctan\left(\frac{\sin\frac{n\theta}{2}}{\cos\frac{n\theta}{2}+(n-1)}\right),
    \qquad
    0\le\theta\le\pi/n.
  \end{equation*}
  One can verify that $g(0)=0$ and, since
  \begin{equation*}
    g'(\theta)
    =
    \frac{2(n-1)(n-2)\sin^2\frac{n\theta}{4}}{\left(\cos\frac{n\theta}{2}+(n-1)\right)^2+\sin^2\frac{n\theta}{2}}
    >
    0,
  \end{equation*}
  we conclude that $g(\theta)>0$ for $0<\theta\le\pi/n$.
\end{proof}

\section{Dynamics of the two-cycle $\{0,\infty\}$}
\label{sec:two_cycle}
As we mentioned earlier, the Laguerre iteration function \eqref{eq:Lp} has a two-cycle $\{0,\infty\}$ for $n\ge3$. We now show that this two-cycle is attracting and its basin of attraction contains a significant portion of $D\cup E$. We will demonstrate later in Section~\ref{sec:boundary} that the basin of attraction can be quite complicated and appears to have a fractal boundary.

The following proposition appears in \cite{curryfiedler88}. It shows that the inner-most and the outer-most white regions in Fig.~\ref{fig:regions} belong to the basin of attraction of the two-cycle. We provide our own proof for the second part of the proposition, as the original proof in \cite{curryfiedler88} is not clear to us.
\begin{proposition}
  \label{prop:0-infty-two-cycle}
  Let $s_0$ be as defined in \eqref{eq:r0s0}. Let $D_{s_0}=\{z\in\C\setminus\{0\}:\ |z|<s_0\}$ and $E_{s_0}=\{z\in\C:\ |z|>1/s_0\}$. Then $L_p(D_{s_0})\subset E_{s_0}$ and $L_p(E_{s_0})\subset D_{s_0}$. Moreover, $D_{s_0}\cup E_{s_0}$ is contained in the basin of attraction of the two-cycle $\{0,\infty\}$.
\end{proposition}
\begin{proof}
  The proof of the first part follows that of \cite{curryfiedler88}. Note first that Theorem~\ref{thm:f_properties} implies that $D_{s_0}\subset D$ and $E_{s_0}\subset E$. Hence, if $z\in D_{s_0}$, then $|L_p(z)|>1/|z|>1/s_0$ and $L_p(z)\in E_{s_0}$. Similarly, if $z\in E_{s_0}$, then $|L_p(z)|<1/|z|<s_0$ and $L_p(z)\in D_{s_0}$.

  To prove the second part, it is enough to show that the basin of attraction of the two-cycle $\{0,\infty\}$ contains $D_{s_0}$. It follows from the previous part that if $z\in D_{s_0}$, then $|L_p(z)|>1/|z|$ and $|L_p^2(z)|<1/|L_p(z)|$, and, consequently, $|L_p^2(z)|<|z|$. Similarly, if $z\in E_{s_0}$, we have $|L_p^2(z)|>|z|$. We will show that for $z\in D_{s_0}$ the even terms of the sequence $\{L_p^k(z)\}$ converge to $0$ and the odd ones to $\infty$. To this end, we observe that $\{|L_p^{2k}(z)|\}$ is a decreasing, bounded, and therefore convergent sequence. If its limit is $0$, we are done, since then $|L_p^{2k}(z)|\to0$ and $|L_p^{2k+1}(z)|>1/|L_p^{2k}(z)|\to\infty$ as $k\to\infty$.

  Assume now that $\lim_{k\to\infty}|L_p^{2k}(z)|=b>0$ and, using the Bolzano--Weierstrass theorem, consider a convergent subsequence of $\{L_p^{2n_k}(z)\}$ and its limit, say, $\tilde{z}\in D_{s_0}$. We then have that if $\tilde{z}$ is not in $\Theta_1$, or if $\tilde{z}\in\Theta_1$ and $\{L_p^{2n_k}(z)\}$ approaches it counter-clockwise, then, by the continuity of $L_p$ and $|L_p|$, we have $b=\lim_{k\to\infty}|L_p^2(L_p^{2n_k}(z))|=|L_p^2(\tilde{z})|<|\tilde{z}|=b$, a contradiction. The only remaining possibility is that $\tilde{z}\in\Theta_1$ and it is not possible to extract a subsequence approaching it counter-clockwise. In that case $\tilde{z}$ is (eventually) approached clockwise, and we can consider a sequence symmetric via a reflection through $\Theta_1$ (Propositions  \ref{prop:rotational_symmetry} and \ref{prop:conjugation_symmetry}). We then obtain a contradiction for this new sequence as in the previous case.
\end{proof}

\begin{remark}
  The regions $D_{s_0}$ and $E_{s_0}$ defined in Proposition~\ref{prop:0-infty-two-cycle} can be seen in Fig.~\ref{fig:regions}: $D_{s_0}$ is the open ball not containing $0$ bounded by the smaller gray annulus, while $E_{s_0}$ is the open region on the outside of the larger gray annulus.
\end{remark}

Although the basin of attraction of the two-cycle $\{0,\infty\}$ is significantly larger than $D_{s_0}\cup E_{s_0}$ (see Section \ref{sec:boundary}), we can immediately extend it in the following sense.
\begin{corollary}
  \label{cor:s0_circles}
  The sets $\{|z|=s_0\}\cap D$ and $\{|z|=1/s_0\}\cap E$ are contained in the basin of attraction of the two-cycle $\{0,\infty\}$.
\end{corollary}
\begin{proof}
From the definitions of $D$ and $E$ in \eqref{eq:regions} it is clear that any point in $\{|z|=s_0\}\cap D$ or $\{|z|=1/s_0\}\cap E$ gets mapped into $D_{s_0}\cup E_{s_0}$, and the claim follows.
\end{proof}

We believe that the points $s_0e^{i\theta_1}$ and $(1/s_0)e^{i\theta_1}$ also converge to the $\{0,\infty\}$ two-cycle. Since these points belong to the set $\partial D\cup\partial E$, their images under the Laguerre iteration map \eqref{eq:Lp} lie on the circles with radii $1/s_0$ and $s_0$, respectively, so it suffices to show that their arguments are different from $\theta_1$. We have not been able to find a simple proof for this statement.

\section{Dynamics of convergence}
\label{sec:convergence}
In this section we state a result \cite{curryfiedler88} that shows that the open annulus bounded by the gray annuli in Fig.~\ref{fig:regions} belongs to the basin of attraction of the roots of $p_n(z)=z^n-1$. Again, it turns out that the basin is actually larger (see Section~\ref{sec:boundary}). We provide an elementary proof of the final statement of the theorem, since in the proof in \cite{curryfiedler88} a reference is made to \cite{bodewig46}, which does not seem to address the claim.

\begin{proposition}
  \label{prop:convergence}
  Let $r_0$ be defined as in \eqref{eq:r0s0}. Let $\hat{K}_0=\{z\in\C:\ r_0<|z|<1\}\subset K_0$ and $\hat{K}_1=\{z\in\C:\ 1<|z|<1/r_0\}\subset K_1$. Then $L_p(\hat{K}_0\cup\hat{K}_1\cup S^1)\subset\hat{K}_0\cup\hat{K}_1\cup S^1$, and the sequence $\{L_p^k(z)\}$ with $z\in\hat{K}_0\cup\hat{K}_1\cup S^1$ converges to a root of $p_n(z)=z^n-1$.
\end{proposition}
\begin{proof}
  We provide a proof along the lines of \cite{curryfiedler88}. First, if $z=re^{i\theta}\in\hat{K}_0\subset K_0$, then $|L_p(z)|<1/r<1/r_0$ from the definition of $K_0$. Using \eqref{eq:mag_Lp2} and $r<1$, we obtain $|L_p(z)|>r>r_0$ (see \cite{curryfiedler88} for details), so we can conclude that $L_p(z)\in\hat{K}_0\cup\hat{K}_1$. In exactly the same fashion, for $z=re^{i\theta}\in\hat{K}_1$ we obtain $r_0<1/r<|L_p(z)|<r<1/r_0$, and, again, $L_p(z)\in\hat{K}_0\cup\hat{K}_1$. In view of Proposition~\ref{prop:unit_circle}, we can conclude that $L_p(z)\in\hat{K}_0\cup\hat{K}_1\cup S^1$ for any $z\in\hat{K}_0\cup\hat{K}_1\cup S^1$.

  We know from Proposition~\ref{prop:unit_circle} that if $z\in S^1$, then $\{L_p^k(z)\}$ converges to a root of $p_n$. It follows from the above inequalities that for $z\in\hat{K}_0\cup\hat{K}_1$ we have
  \begin{equation}
    \label{eq:K0K1_inequalities}
    \min\{|z|,1/|z|\}<|L_p(z)|<\max\{|z|,1/|z|\},
    \qquad
    \min\{|z|,1/|z|\}<\frac{1}{|L_p(z)|}<\max\{|z|,1/|z|\},
  \end{equation}
  and, consequently, the sequence $\left\{\left||L_p^k(z)|-\dfrac{1}{|L_p^k(z)|}\right|\right\}$ is decreasing and convergent. This sequence converges to $0$ (and $\lim_{k\to\infty}|L_p^k(z)|=1$), since otherwise we can consider a subsequence (not relabeled) such that $|L_p^k(z)|\to b\ne1$, extract a further subsequence such that $L_p^{n_k}(z)\to\tilde{z}$, and argue as in the proof of Proposition~\ref{prop:0-infty-two-cycle} that $|L_p(\tilde{z})|=|\tilde{z}|=b$, contradicting the inequalities in \eqref{eq:K0K1_inequalities}.

  Finally, for $z\in\hat{K}_0\cup\hat{K}_1$ and the sequence $\{L_p^k(z)\}$, consider a convergent subsequence $\{L_p^{n_k}(z)\}$ and its limit $\tilde{z}\in S^1$. If $\tilde{z}\in S^1\setminus\Theta_1$, then the sequence $\{L_p^j(\tilde{z})\}$ converges to a root $z^*$ of $p_n$ by Proposition~\ref{prop:unit_circle}. By the continuity of $L_p$
  and the fact that $\{L_p^j(\tilde{z})\}$ converges to $z^*$ monotonically in the sense of Proposition~\ref{prop:unit_circle},
  we have $\lim_{k\to\infty}L_p^{n_k+j}(z)\to L_p^j(\tilde{z})$ for any $j\ge0$. This implies that there exist a large enough $k$ and a large enough $j$ such that $L_p^{n_k+j}(z)$ is in the basin of attraction of the root $z^*$, and, therefore, the whole sequence $\{L_p^k(z)\}$ converges to $z^*$. In particular, $\tilde{z}=z^*$.

  The remaining case with the limit of the subsequence $\{L_p^{n_k}(z)\}$ satisfying $\tilde{z}\in S^1\cap\Theta_1$ can be treated as in the proof of Proposition~\ref{prop:0-infty-two-cycle} by considering further subsequences approaching $\tilde{z}$ clockwise or counter-clockwise; in either case we obtain a contradiction, since arguing as in the previous paragraph we conclude that $\tilde{z}$ has to be a root of $p_n$.
\end{proof}

We again have an extension of the above proposition, arguing as in the proof of Corollary~\ref{cor:s0_circles}.
\begin{corollary}
  The sets $\{|z|=r_0\}\cap K_0$ and $\{|z|=1/r_0\}\cap K_1$ are contained in the basin of attraction of the roots of unity.
\end{corollary}

However, the remaining points on the circles $\{|z|=r_0\}$ and $\{|z|=1/r_0\}$ form non-trivial, finite two-cycles \cite{ray66,curryfiedler88}.
\begin{proposition}
  \label{prop:two-cycle}
  For every $\theta_0\in\Theta_0$, the set $\{r_0e^{i\theta_0},(1/r_0)e^{i\theta_0}\}$ with $r_0$ defined in \eqref{eq:r0s0} is a two-cycle for the Laguerre iteration function \eqref{eq:Lp}.
\end{proposition}
\begin{proof}
  By Proposition \ref{prop:rotational_symmetry}, we can assume $\theta_0=0$. From \eqref{eq:Lp} we have that $L_p$ maps real, positive numbers to real, positive numbers, so $L_p(r_0)=1/r_0$ since $r_0\in\partial D$. Similarly, $L_p(1/r_0)=r_0$ since $1/r_0\in\partial E$.
\end{proof}

For completeness, we state the following result that completes the dynamics on $\Theta_0$.
\begin{proposition}
  For every $\theta_0\in\Theta_0$, the set $\{re^{i\theta_0}:\ 0<r<r_0 \text{ or }1/r_0<r\}$ belongs to the basin of attraction of the two-cycle $\{0,\infty\}$ for the Laguerre iteration function \eqref{eq:Lp}.
\end{proposition}
\begin{proof}
  The proof is similar to the proof of the second part of Proposition \ref{prop:0-infty-two-cycle}. For $0<r<r_0$, we get $0<L_p^2(r)<r$, so $\lim_{k\to\infty}L_p^{2k}(r)=\tilde{r}$ with $0\le\tilde{r}<r_0$. Now, if $\tilde{r}>0$, we would have the contraction $L_p^2(\tilde{r})<\tilde{r}$ and we would also obtain $L_p^2(\tilde{r})=\tilde{r}$ by continuity of $L_p$. Therefore, $\tilde{r}=0$, and the rest follows by the symmetries of the iteration function.
\end{proof}

Finally, the following result clearly demonstrates that Laguerre's method is not suitable for finding roots of unity for large-degree polynomials \cite{ray66}.
\begin{proposition}
  The set of all points in $\C$ for which Laguerre's method \eqref{eq:method} converges to a root of $p_n(z)=z^n-1$, $n\ge5$, is contained in the annulus $\{z\in\C:\ s_0\le|z|\le1/s_0\}$ with $s_0$ defined in \eqref{eq:r0s0}, whose measure tends to $0$ as $n\to\infty$.
\end{proposition}
\begin{proof}
  The claim follows from Proposition \ref{prop:0-infty-two-cycle} and Theorem \ref{thm:f_properties}, part (1).
\end{proof}

\section{The regions of convergence and their boundaries}
\label{sec:boundary}
In this section we present primarily computational results that address the structure of the basins of attraction of Laguerre's method \eqref{eq:method} applied to $p_n(z)=z^n-1$ with $n\ge5$. These results raise additional questions that we summarize in the next section.

We mentioned earlier that for $n=2$ it takes one iteration to get to a root of $p_2$ from any initial guess. It is also known that for $n=3$ or $4$ the method is globally convergent to a root of $p_n$ \cite{drakopoulos03,curryfiedler88,ray66}. However, from the above analysis it follows that for $n\ge5$ this is no longer true; more specifically, the basin of attraction for each $n\ge5$ is contained in the annulus $\{s_0\le|z|\le1/s_0\}$ with $s_0$ given in \eqref{eq:r0s0}. Since $s_0$ is not easily computable, we can use the upper bound $1/s_0<(n-1)^{2/(n-4)}$ (see Theorem \ref{thm:f_properties}). In Fig.~\ref{fig:basins}, we present examples of the basins of attraction for $n=5$, $8$, $12$, and $16$. These are plotted in the squares $[-(n-1)^{2/(n-4)},(n-1)^{2/(n-4)}]\times[-(n-1)^{2/(n-4)},(n-1)^{2/(n-4)}]$ and show that the upper bound is a good estimate of $1/s_0$. Note how, in accordance with Theorem~\ref{thm:f_properties}, the region of convergence shrinks as $n$ increases.
\begin{figure}
  \includegraphics[width=0.49\textwidth]{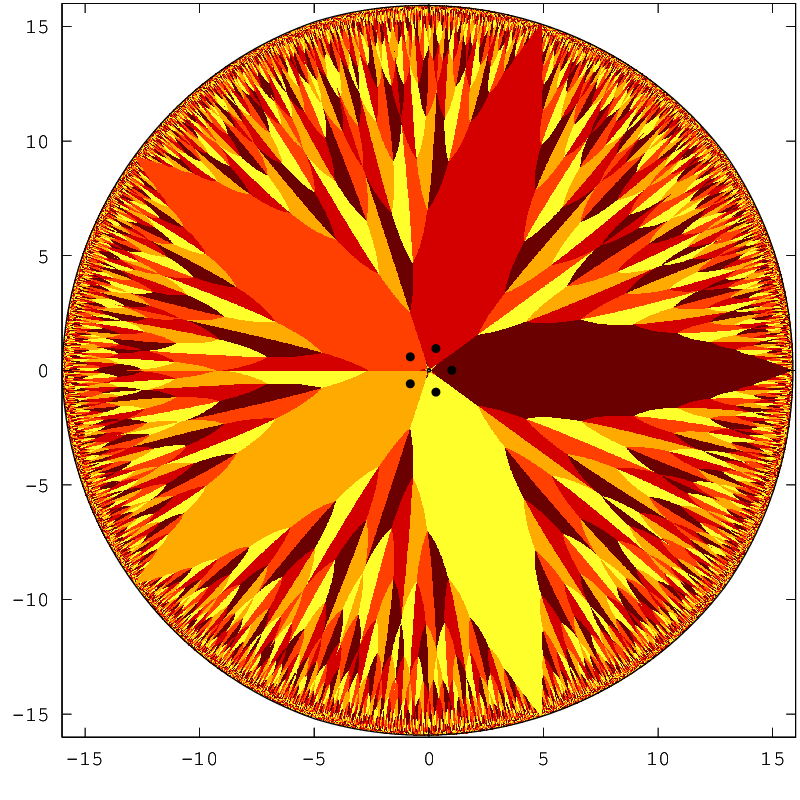}
  \hfill
  \includegraphics[width=0.49\textwidth]{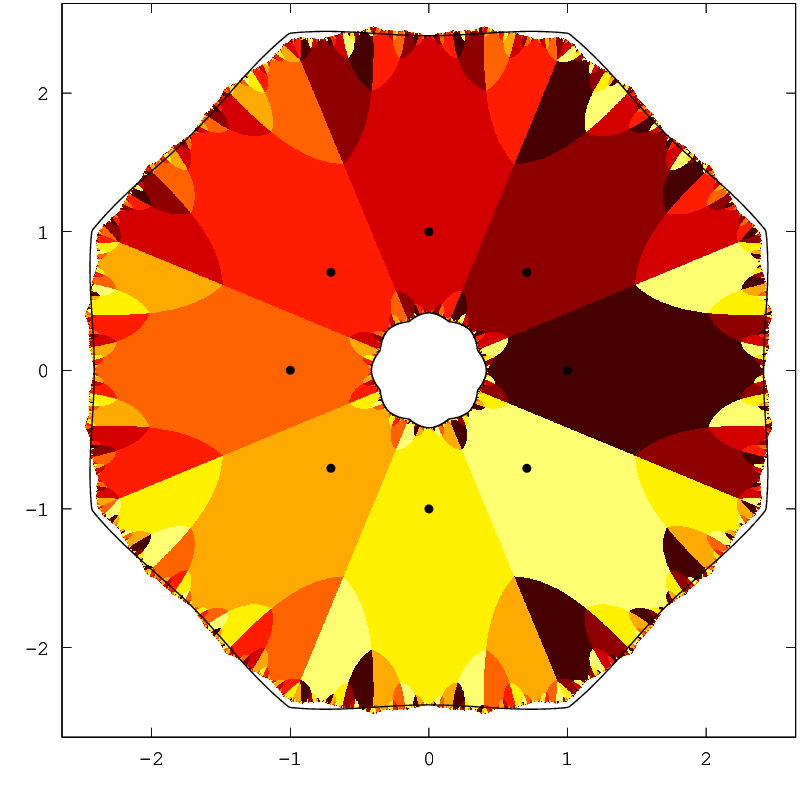}
  \\
  \includegraphics[width=0.49\textwidth]{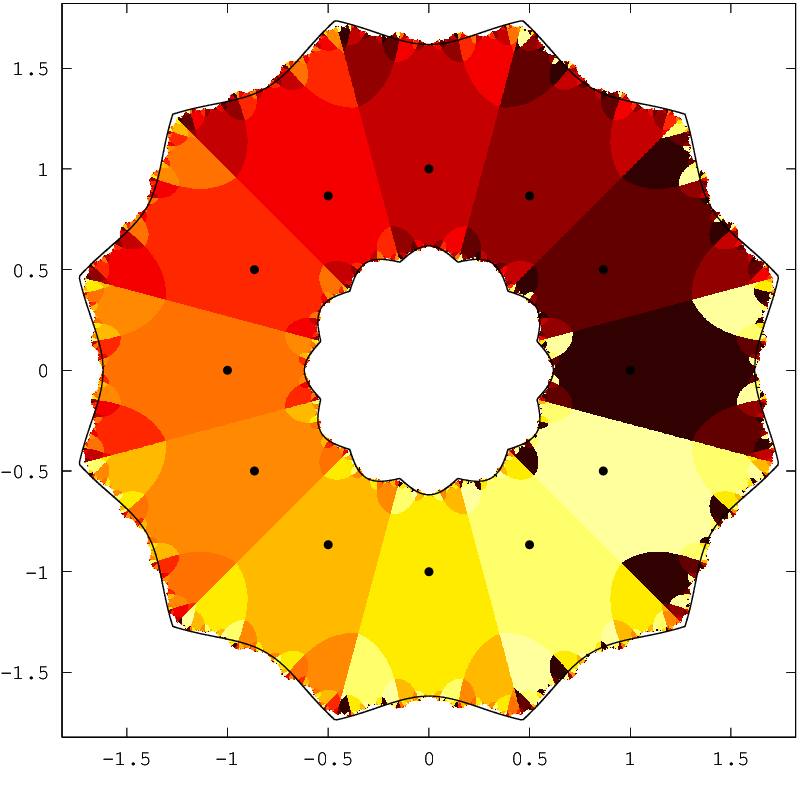}
  \hfill
  \includegraphics[width=0.49\textwidth]{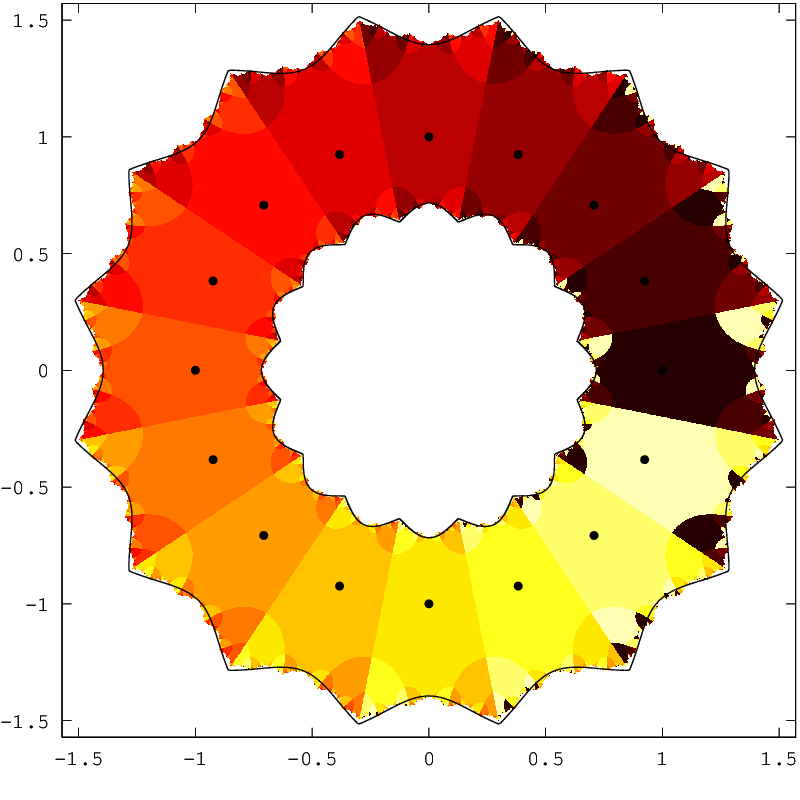}
  \caption{Numerically computed basins of attraction of Laguerre's method applied to the polynomials $p_n(z)=z^n-1$ with $n=5$, $8$, $12$, and $16$ (row-wise, left to right). Each color corresponds to a basin of attraction of a root in the basin. The two black curves in each image are $\partial D$ and $\partial E$, and the dots represent the roots of $p_n$. Note how the boundary of the basin of attraction tracks $\partial D\cup\partial E$, but it appears fractal.}
  \label{fig:basins}
\end{figure}

The boundary of the region of convergence appears fractal (see, e.g., \cite{falconer03} for more on fractals). We demonstrate this in Fig.~\ref{fig:fractal_bdry}, where we show parts of the external boundary of the regions of convergence in the sectors with $\pi(n-1)/n<\theta<\pi(n+1)/n$ for $n=8$, $16$, $24$, and $32$. By the rotational symmetry, Proposition \ref{prop:rotational_symmetry}, the other parts of the external boundary are congruent to the displayed ones.
\begin{figure}
  \includegraphics[height=0.35\textheight]{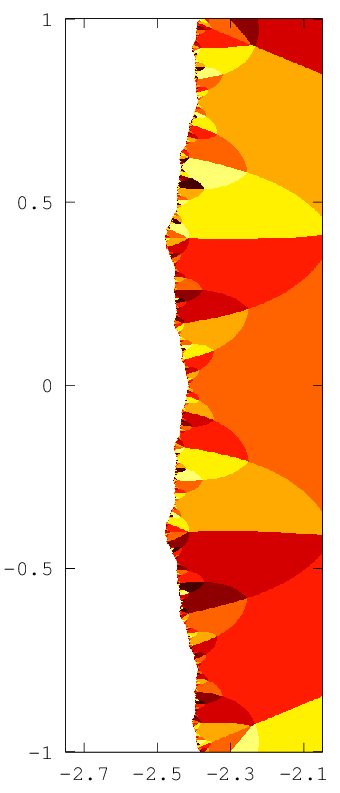}
  \hfill
  \includegraphics[height=0.35\textheight]{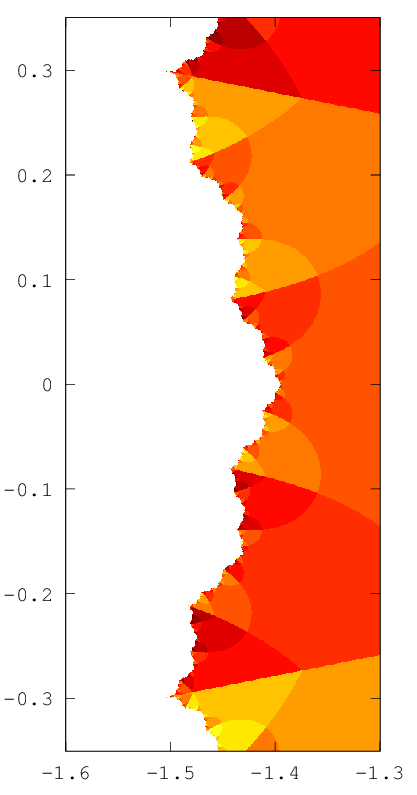}
  \hfill
  \includegraphics[height=0.35\textheight]{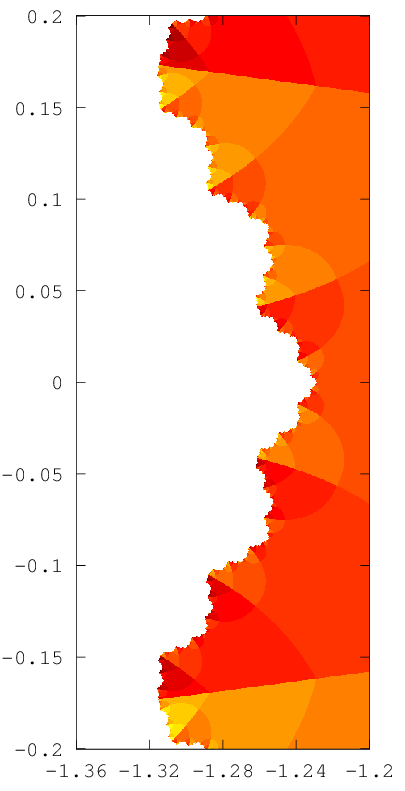}
  \hfill
  \includegraphics[height=0.35\textheight]{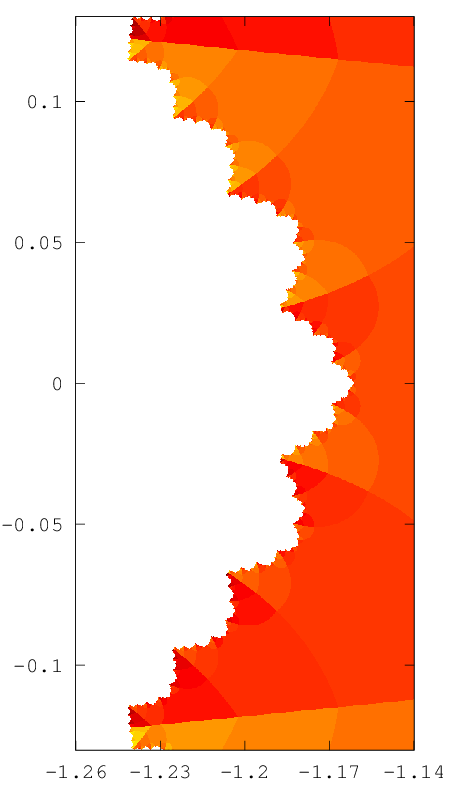}
  \caption{Parts of the boundary of the region of convergence to a root of $p_n(z)=z^n-1$ for $n=8$, $16$, $24$, and $32$. They correspond to the sectors with $\pi(n-1)/n<\theta<\pi(n+1)/n$ and demonstrate the fractal behavior of the boundary.}
  \label{fig:fractal_bdry}
\end{figure}

The boundaries displayed in Fig.~\ref{fig:fractal_bdry} appear self-similar, but they are only quasi self-similar (see, e.g., \cite{falconer03,mclaughlin87} for more on quasi self-similarity). We demonstrate this observation in Fig.~\ref{fig:quasi-ss-bdry}, where we can see slight changes of shape as we zoom in and also as we more carefully examine the shapes within each figure.
\begin{figure}
  \includegraphics[width=\textwidth]{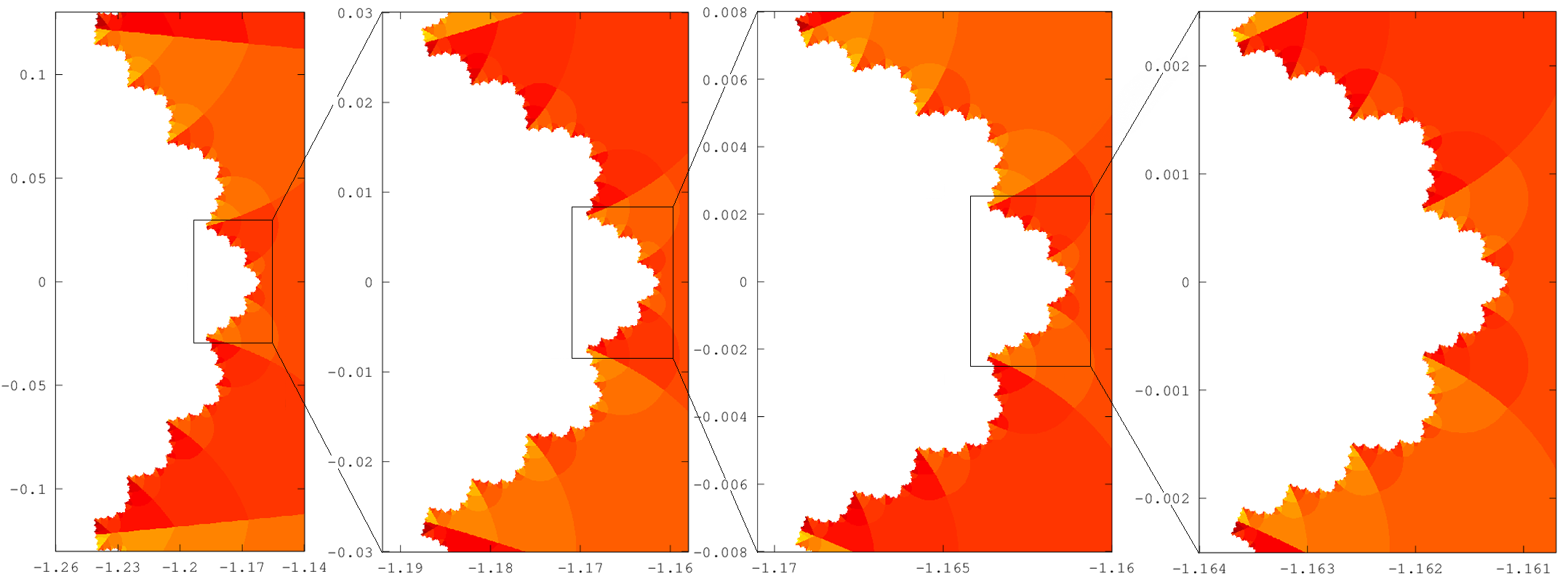}
  \caption{Three consecutive zoom levels into a part of the boundary for $n=32$ (shown also in Fig.~\ref{fig:fractal_bdry}) clearly demonstrate that the boundary of the region of convergence is not self-similar, only quasi self-similar.}
  \label{fig:quasi-ss-bdry}
\end{figure}

In addition, and it came to us as quite a surprise, it seems that the regions of convergence as shown in color in Fig.~\ref{fig:basins} are not, in general, (disregarding the ``hole'' in the middle) simply connected, or even connected! In Figs.~\ref{fig:holes_zoom_128} and \ref{fig:holes_zoom_1024} we present results with $n=128$ and $n=1024$, respectively, and several consecutive zooms into the ``gray area'' $\{1/r_0<|z|<1/s_0\}$ shown in Fig.~\ref{fig:regions}. Note the intricate structure that becomes more prominent for larger values of $n$. Both figures clearly demonstrate the disconnectedness of the basin of attraction of the roots of $p_n$. We chose the values of $n=128$ and $n=1024$ since the ``holes'' become detectable with a naked eye around $n=120$ and we could zoom into them, and the larger value to demonstrate how much more the structure develops as $n$ increases.
\begin{figure}
  \includegraphics[height=0.88\textheight]{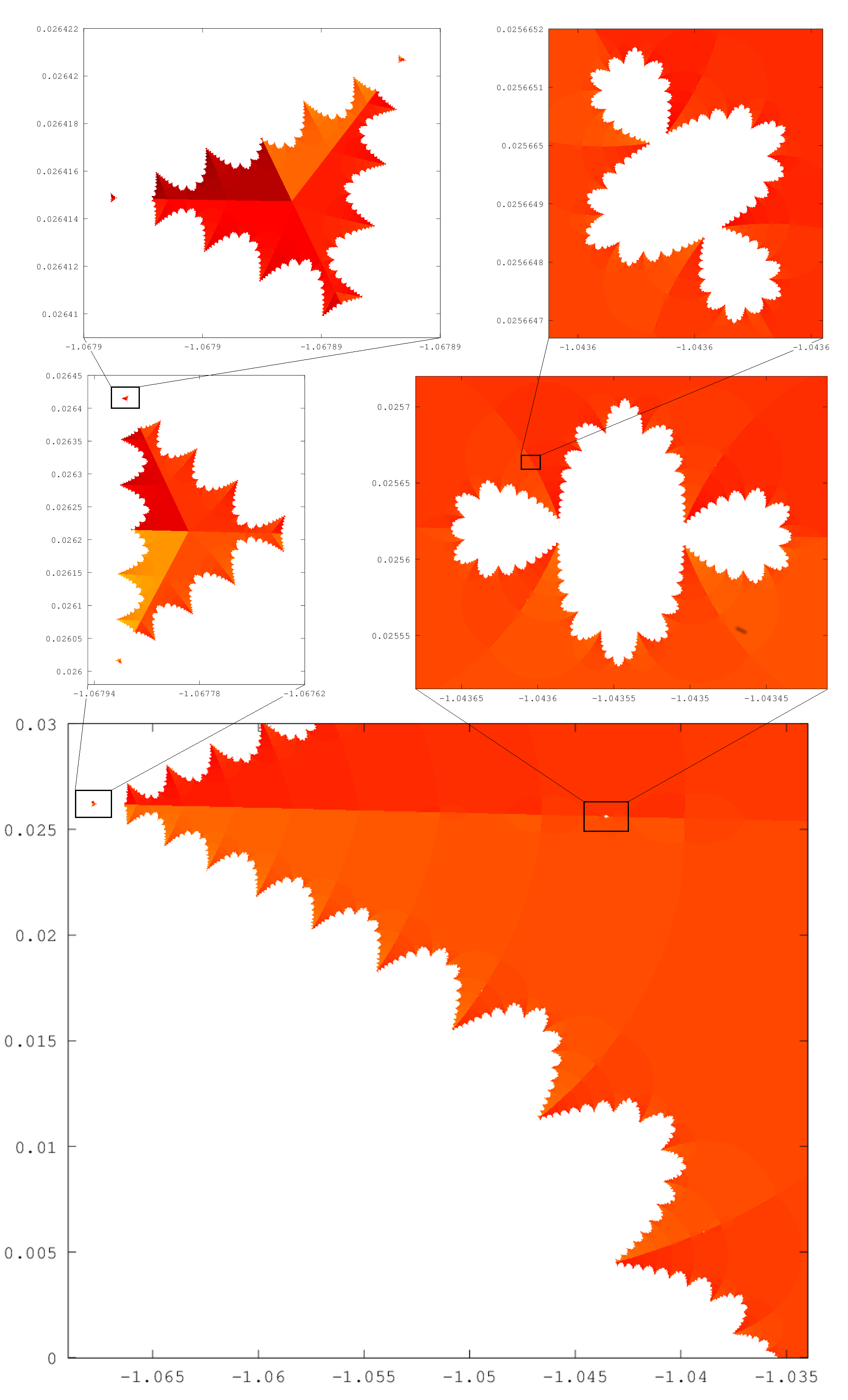}
  \caption{Several consecutive zooms into two parts of the boundary for $n=128$. Note that the region of convergence is not either connected, nor simply connected. In particular, it appears that the basin of attraction of the roots consists of infinitely many (quasi) self-similar disconnected sets (zooms on the left), and infinitely many (quasi) self-similar ``holes'' corresponding to basins of attraction of the two-cycle $\{0,\infty\}$ (zooms on the right).}
  \label{fig:holes_zoom_128}
\end{figure}
\begin{figure}
  \includegraphics[width=\textwidth]{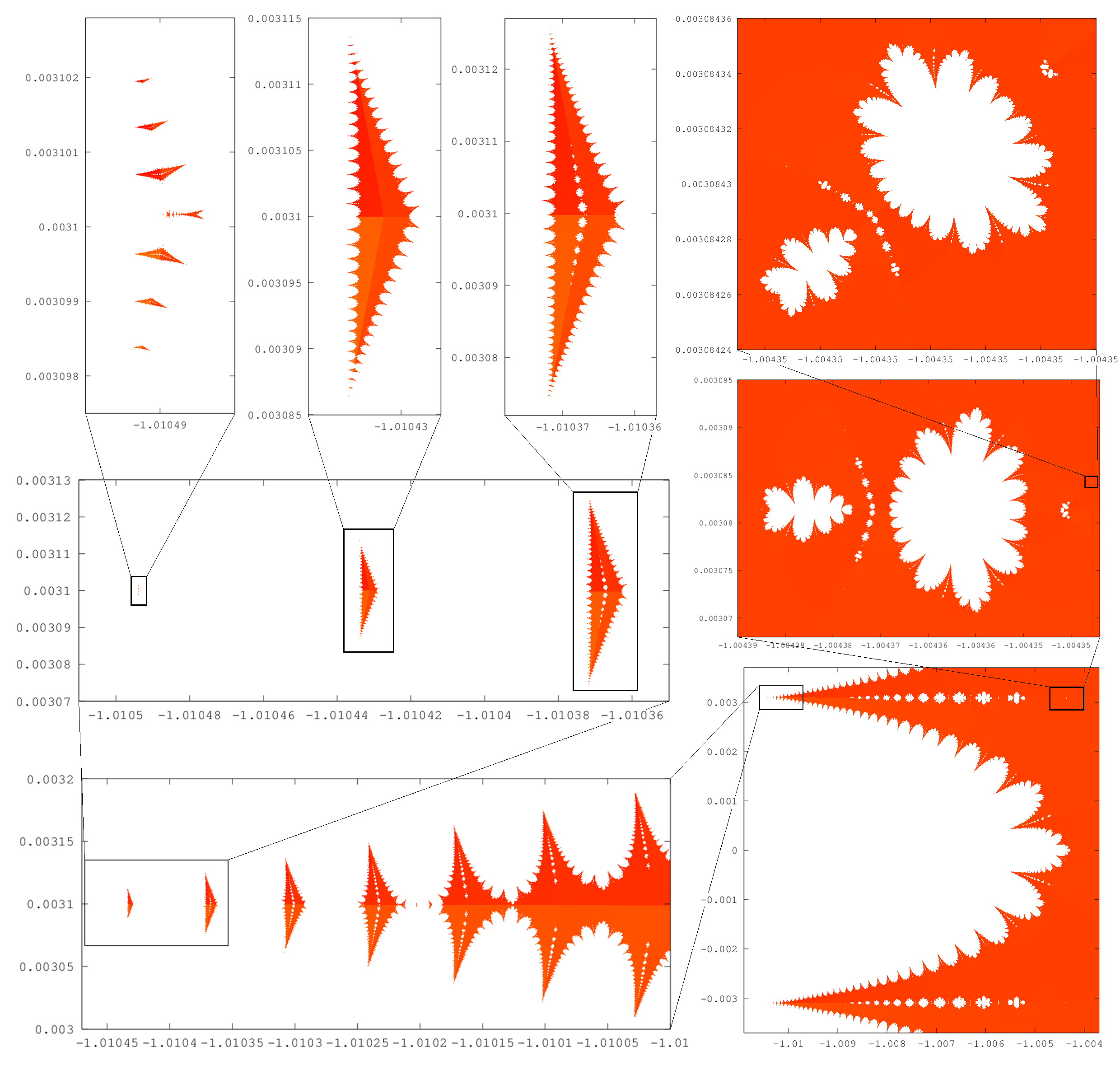}
  \caption{Several consecutive zooms into two parts of the boundary for $n=1024$. Much more structure and disconnectedness becomes visible compared to the case with $n=128$ (Fig.~\ref{fig:holes_zoom_128}). Note how the regions of convergence to the two-cycle $\{0,\infty\}$ (in white) extend through the ``gray areas'' $1/r_0<|z|<1/s_0$, and also note the (quasi) self-similarity throughout.}
  \label{fig:holes_zoom_1024}
\end{figure}

\section{Conclusions and outstanding questions}
\label{sec:conclusions}
In the previous sections we have analyzed the behavior of Laguerre's method applied to the polynomials $p_n(z)=z^n-1$ in the extended complex plane and provided computational results demonstrating the interesting behavior of the method. We now have an almost complete understanding of the behavior of the method outside of the two gray areas that contain the boundary of the region of convergence. We concluded that for initial guesses in $\hat{K}_0\cup\hat{K}_1\cup S^1$ the method converges to a root of $p_n$ (Proposition~\ref{prop:convergence}), and for initial guesses in $D_{s_0}\cup E_{s_0}$ the method converges to the two-cycle $\{0,\infty\}$ (Proposition~\ref{prop:0-infty-two-cycle}).

The numerical results indicate that the basin of attraction of the roots and the basin of attraction of the two-cycle share a common boundary, which should then be an invariant set under the Laguerre iteration function \eqref{eq:Lp} and consist only of finite cycles and infinite orbits. We have not pursued this direction in great depth, as it would likely require extending the theory of Julia and Fatou sets \cite{milnor06} to functions that are not rational. Note that the Laguerre iteration function \eqref{eq:Lp} is not rational even if $n$ is even due to the choice of sign in the denominator of the method. We have, however, attempted to find some short cycles, other than those given in Proposition \ref{prop:two-cycle}, numerically in the following way. First, we used the computational software program {\em Mathematica} to generate the contour plots of $\Re\left(L_p^k(z)-z\right)=0$ and $\Im\left(L_p^k(z)-z\right)=0$ in the sector $0<\theta<\pi/n$ (recall the symmetries Propositions \ref{prop:rotational_symmetry}--\ref{prop:inversion_symmetry}), and used the visually discovered points of intersection as initial guesses in root-finding algorithms for $L_p^k(z)-z$. This way we have been able to find some $2$-, $4$-, and $6$-cycles for polynomials of low degrees. In particular, it appears that $2$-cycles in the sector $0<\theta<\pi/n$ only exist for $n\ge10$ with $p_{10}$--$p_{16}$ having one such $2$-cycle each; $p_{17}$--$p_{26}$ having two; $p_{27}$--$p_{38}$ having three, etc. Regarding $4$-cycles, we found two for $n=5,6,7$; four for $n=8$; eight for $n=9$; nine for $n=10$; ten for $n=11$ and $12$, etc. Finally, $6$-cycles appear to start at $n=6$; we found ten of them for $n=6$, twelve for $n=7$, and twenty-three for $n=8$. Not surprisingly, we haven't found any short odd-cycles, which seems reasonable due to the expected behavior of points near $\partial D$ getting mapped close to $\partial E$ and vice versa. We list the found $4$- and $6$-cycles for $n=5,6,7,8$ in Table \ref{tab:cycles}, where all numbers have been computed to $16$ significant digit accuracy.
\begin{table}
  \centering
  \caption{Period-four and period-six cycles in the sectors $0<\theta<\pi/n$ for $n=5,6,7,8$.}
  \begin{tabular}{|c||c|c|}
    \hline
    $n$ & Four-cycles in $0<\theta<\pi/n$ & Six-cycles in $0<\theta<\pi/n$ \\ \hline\hline
    $5$ & $14.76136221056119 + 6.053684491748273i$  &                                            \\ \hline
    $5$ & $13.34758676939078 + 8.758987500188936i$  &                                            \\ \hline\hline
    $6$ & $4.749579144551457 + 1.098207699050568i$  & $4.809680273550060 + 0.2938473062700105i$  \\ \hline
    $6$ & $4.462144769610253 + 2.042313839245265i$  & $4.807845850061632 + 0.5532613795970850i$  \\ \hline
    $6$ &                                           & $4.791479366250020 + 0.7926029359282372i$  \\ \hline
    $6$ &                                           & $4.758607318036074 + 1.041475486533466i$   \\ \hline
    $6$ &                                           & $4.703162128954519 + 1.307935002359033i$   \\ \hline
    $6$ &                                           & $4.660756967199228 + 1.487046196784297i$   \\ \hline
    $6$ &                                           & $4.587335089636846 + 1.730901724312133i$   \\ \hline
    $6$ &                                           & $4.504230549742958 + 1.940512624824430i$   \\ \hline
    $6$ &                                           & $4.394056696857116 + 2.184780604963821i$   \\ \hline
    $6$ &                                           & $4.271115104571219 + 2.408180892551038i$   \\ \hline\hline
    $7$ & $3.102711305833646 + 0.4791536373837452i$ & $3.086956249849817 + 0.09918082640500742i$ \\ \hline
    $7$ & $3.005076854712194 + 1.041210971892816i$  & $3.098674298771984 + 0.2141440467880112i$  \\ \hline
    $7$ &                                           & $3.108949745463634 + 0.2853986438548191i$  \\ \hline
    $7$ &                                           & $3.105161257066959 + 0.3575768153296901i$  \\ \hline
    $7$ &                                           & $3.103107584356313 + 0.4651267058741789i$  \\ \hline
    $7$ &                                           & $3.096111405599539 + 0.5516620616566748i$  \\ \hline
    $7$ &                                           & $3.104755891427699 + 0.6231011218949135i$  \\ \hline
    $7$ &                                           & $3.069516544423272 + 0.7787753426450393i$  \\ \hline
    $7$ &                                           & $3.013630591486307 + 1.011069152189774i$   \\ \hline
    $7$ &                                           & $2.988317222029729 + 1.090657585339925i$   \\ \hline
    $7$ &                                           & $2.947434684423189 + 1.193708033582317i$   \\ \hline
    $7$ &                                           & $2.917334700654045 + 1.353477443986847i$   \\ \hline\hline
    $8$ & $2.452675491472578 + 0.2785857657502200i$ & $2.424861149787357 + 0.04867681760903371i$ \\ \hline
    $8$ & $2.475125064051807 + 0.4260510136323984i$ & $2.433534405337240 + 0.07300790020468901i$ \\ \hline
    $8$ & $2.419209618854812 + 0.6711649955368295i$ & $2.436567623533557 + 0.1147077080457789i$  \\ \hline
    $8$ & $2.403834152721369 + 0.9295918014456642i$ & $2.443481963307878 + 0.1425019298788600i$  \\ \hline
    $8$ &                                           & $2.453807739243897 + 0.1734183047708730i$  \\ \hline
    $8$ &                                           & $2.448811352673064 + 0.2130028821893788i$  \\ \hline
    $8$ &                                           & $2.452347914830289 + 0.2731330464450077i$  \\ \hline
    $8$ &                                           & $2.452422308102194 + 0.3140142922891776i$  \\ \hline
    $8$ &                                           & $2.460907150206826 + 0.3413782443508301i$  \\ \hline
    $8$ &                                           & $2.475788631367756 + 0.3743161467144843i$  \\ \hline
    $8$ &                                           & $2.478533197214976 + 0.3940810991477981i$  \\ \hline
    $8$ &                                           & $2.474834833464490 + 0.4192210929378601i$  \\ \hline
    $8$ &                                           & $2.452717379844995 + 0.4708078573787699i$  \\ \hline
    $8$ &                                           & $2.445815983913505 + 0.5121550600972321i$  \\ \hline
    $8$ &                                           & $2.446892544221102 + 0.5361638213148769i$  \\ \hline
    $8$ &                                           & $2.421772891898377 + 0.6579277753771876i$  \\ \hline
    $8$ &                                           & $2.413249889360346 + 0.6953853836452080i$  \\ \hline
    $8$ &                                           & $2.395254878356101 + 0.7555517867904204i$  \\ \hline
    $8$ &                                           & $2.393829791477975 + 0.8012023541812359i$  \\ \hline
    $8$ &                                           & $2.395554660155073 + 0.8271069757313417i$  \\ \hline
    $8$ &                                           & $2.405037612365332 + 0.9182940336703574i$  \\ \hline
    $8$ &                                           & $2.396106513659878 + 0.9454683171801179i$  \\ \hline
    $8$ &                                           & $2.394665948621691 + 0.9684820140915532i$  \\ \hline
  \end{tabular}
  \label{tab:cycles}
\end{table}

Many questions remain. What is the shape of the boundary of the region of convergence? We see in Fig.~\ref{fig:basins} that the boundary appears to track $\partial D$ and $\partial E$, but it does not coincide with these sets. The boundary is fractal (Figs.~\ref{fig:fractal_bdry} and \ref{fig:quasi-ss-bdry}) and, moreover, has many other components in the annuli determined by $r_0$ and $s_0$ (Figs.~\ref{fig:holes_zoom_128} and \ref{fig:holes_zoom_1024}). It may be of interest to see whether a fractal dimension of the boundary has a simple dependence on $n$. We speculate that the dimension might grow from $1$ to $2$ as $n$ increases from $5$ to $\infty$, but we have not pursued this idea further.

It would also be interesting to see if other families of polynomials exhibit similar features to those observed for $p_n(z)=z^n-1$. In particular, what determines the size and shape of the regions of convergence to the roots? Is it due to the symmetry of the roots that the measure of the regions of convergence tends to zero? If so, would other symmetric arrangements of the roots yield similar results? Perhaps the questions should be reversed. Are there families of polynomials for which Laguerre's method converges to a root except if starting from a set of zero measure? If so, what are they? We intend to look into some of these questions in future work.

%
%

We conclude with the following interesting observation. The fact that the method theoretically converges only in the small annulus in the neighborhood of the unit circle $S^1$ suggests that Laguerre's method is unsuitable practically and raises a valid concern for general polynomials. On the other hand, when the method is implemented in its general formulation \eqref{eq:laguerre} and applied to polynomials $p_n(z)=z^n-1$, the resulting image of the basins of attraction may look like Fig.~\ref{fig:basin_8_chaos}, in which the polynomial $p_8(z)=z^8-1$ is used and the basins of attraction are computed on a $1000\times1000$ grid of points.
\begin{figure}
  \includegraphics[width=0.32\textwidth]{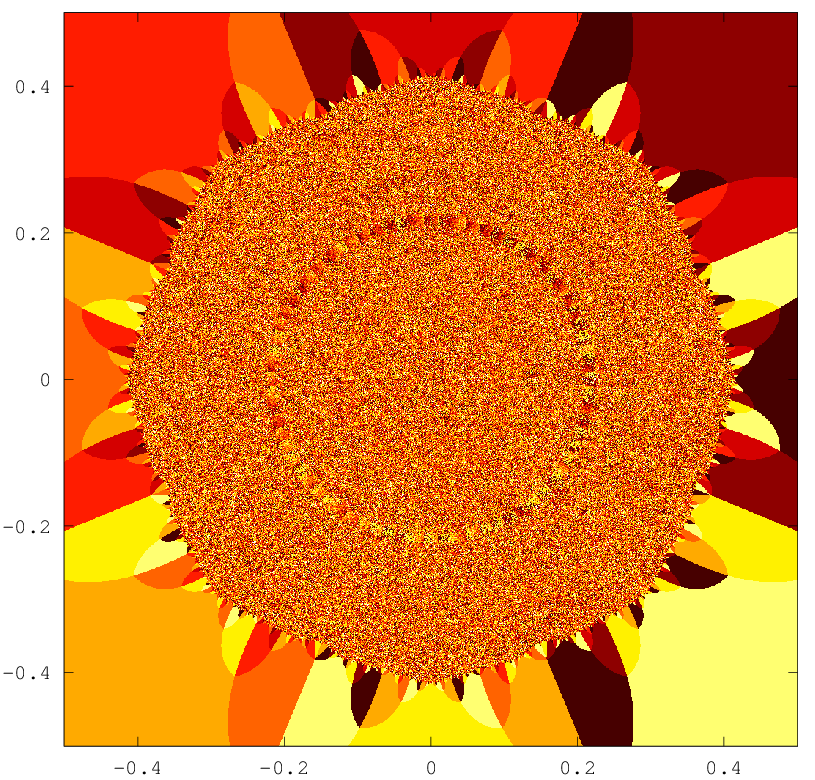}
  \hfill
  \includegraphics[width=0.32\textwidth]{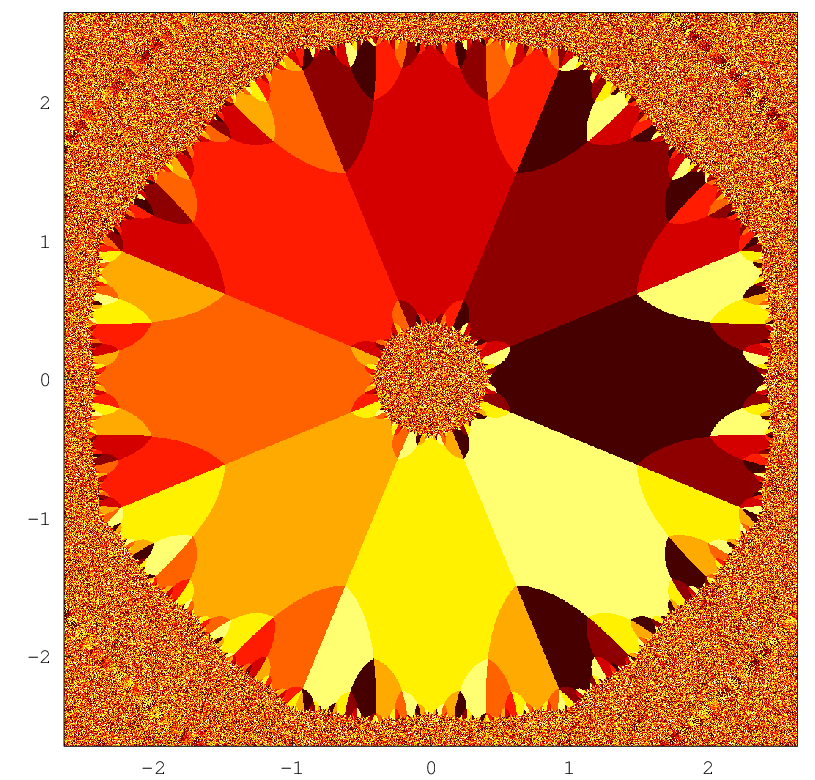}
  \hfill
  \includegraphics[width=0.32\textwidth]{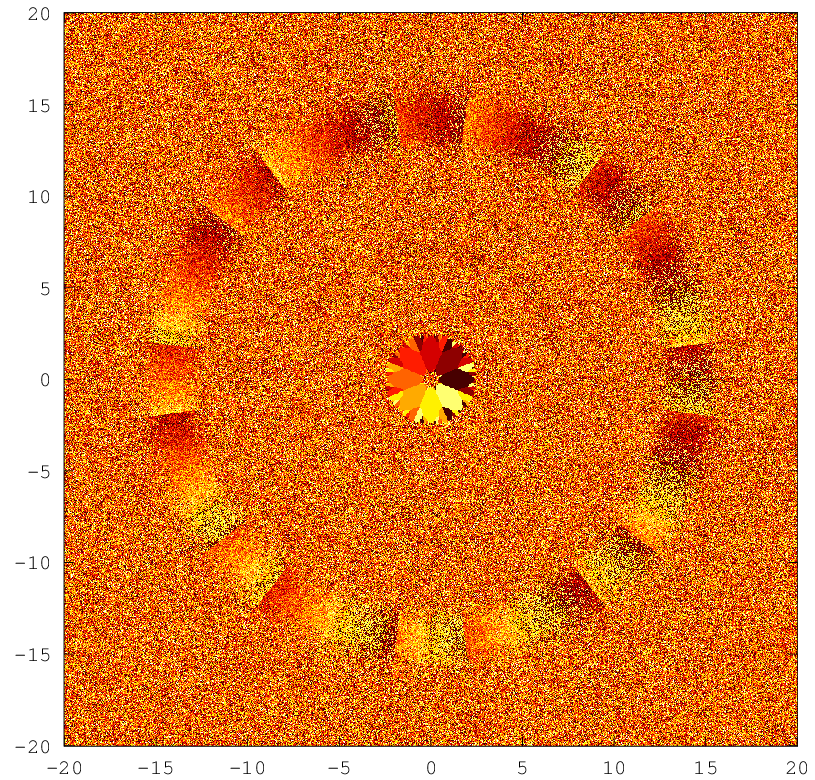}
  \caption{Three levels of zoom into the computed basin of attraction for $p_8(z)=z^8-1$ using the general formulation of the method \eqref{eq:laguerre}. If the iterations are allowed to run to convergence (or a prescribed maximum number of iterations, $100$ in this computation), we observe chaotic convergence to roots even for initial guesses starting in the region of theoretical convergence to the two-cycle $\{0,\infty\}$ (compare to Fig.~\ref{fig:basins}, where the basin of attraction of $\{0,\infty\}$ is colored white). This behavior is due to the loss of significance in the computation. The middle figure is the same as in Fig.~\ref{fig:basins} (upper right), the left figure is a zoom into the center part of the middle figure, and the right figure is a zoom out to a $40\times40$ square.}
  \label{fig:basin_8_chaos}
\end{figure}
Note that visually the method converges from any point in the displayed squares, which is not the case when the formulation \eqref{eq:Lp} is used. The reason for this behavior is the loss of significance in the computation of the expression $(n-1)^2\left(p'(z)\right)^2-n(n-1)p(z)p''(z)$ in the denominator of \eqref{eq:laguerre}. Note that both terms in the difference have leading terms $n^2(n-1)^2z^{2n-2}$, and the actual difference should be equal to $n^2(n-1)^2z^{n-2}$. We therefore see that, for large $|z|$, significant errors will occur in the computation of the square root in \eqref{eq:laguerre}. In fact, the relative error in the computation of the square root is roughly proportional to $\sqrt{\varepsilon(1+|z|^n)}$, where $\varepsilon$ is the machine epsilon, so with $n=8$ and the usual $64$-bit double precision, the relative error is on the order of $1$ with $|z|$ as small as $100$. We note that the loss of significance will occur for any polynomial $p(z)$ of degree $n$ and $|z|$ large enough, since for a general polynomial of degree $n$ the difference $(n-1)^2\left(p'(z)\right)^2-n(n-1)p(z)p''(z)$ will have a leading term of order $z^{2n-4}$, two orders of magnitude smaller than the leading terms  of $(n-1)^2\left(p'(z)\right)^2$ and $n(n-1)p(z)p''(z)$. Perhaps this observation helps explain the popular notion that Laguerre's method seems to converge to a root from almost any initial guess.

\bibliography{laguerre}
\bibliographystyle{abbrv}

\end{document}